\documentclass{article}

\usepackage{amsmath,amssymb,amsthm}
\usepackage{tikz}
\usepackage{color}
\usepackage[toc]{appendix}
\usepackage{graphicx}
\usepackage{fancyhdr}
\usepackage{enumitem}
\usepackage{bbm}
\usepackage{parskip}
\usepackage{float}
\usepackage{chngpage}
\usepackage{calc}
\usepackage{bigints}
\usepackage{array}
\usepackage{booktabs}
\usepackage{rotating}
\usepackage{multirow}
\usepackage{adjustbox}
\usepackage{tabularx}
\usepackage{verbatim}
\usepackage{mathtools}
\usepackage{ragged2e}
\usepackage[makeroom]{cancel}
\usepackage{caption}
\usepackage{hyperref}
\usepackage{caption}
\usepackage{subcaption}
\usepackage{appendix}
\usepackage{pgfplots}
\pgfplotsset{compat=1.16}
\usepackage[english]{babel}
\usepackage{hyphenat}
\usepackage[makeindex]{imakeidx}
\usetikzlibrary{datavisualization}
\usetikzlibrary{matrix}
\usetikzlibrary{datavisualization.formats.functions}

\newtheorem{theorem}{Theorem}[section]
\newtheorem{proposition}[theorem]{Proposition}

\newtheorem{lemma}[theorem]{Lemma}

\newtheorem{remark}[theorem]{Remark}
\newtheorem{assumption}[theorem]{Assumption}

\makeatletter
\def\section{\@startsection {section}{1}{\z@}{3.25ex plus 1ex minus
		.2ex}{1.5ex plus .2ex}{\large\bf}}
\def\subsection{\@startsection{subsection}{2}{\z@}{3.25ex plus 1ex minus
		.2ex}{1.5ex plus .2ex}{\normalsize\bf}}
\@addtoreset{equation}{section} 
\makeatother

\title{Probabilistic derivation and analysis of the chemical diffusion master equation with mutual annihilation}

\author{Alberto Lanconelli\thanks{Dipartimento di Scienze Statistiche Paolo Fortunati, Università di Bologna, Bologna, Italy. \textbf{e-mail}: alberto.lanconelli2@unibo.it} \and Berk Tan Perçin \thanks{Dipartimento di Scienze Statistiche Paolo Fortunati, Università di Bologna, Bologna, Italy. \textbf{e-mail}: berktan.perçin2@unibo.it} }

\date{\today}

\begin{document}
	
	\maketitle
	
	\bigskip
	
	\begin{abstract}
		We propose a probabilistic derivation of the so-called chemical diffusion master equation (CDME) and describe an infinite dimensional moment generating function method for finding its analytical solution. CDMEs model by means of an infinite system of coupled Fokker-Planck equations the probabilistic evolution of chemical reaction kinetics associated with spatial diffusion of individual particles; here, we focus an creation and mutual annihilation chemical reactions combined with Brownian diffusion of the single particles. Our probabilistic approach mimics the derivation of backward Kolmogorov equations for birth-death continuous time Markov chains. Moreover, the proposed infinite dimensional moment generating function method links certain finite dimensional projections of the solution of the CDME to the solution of a single linear fourth order partial differential equation containing as many variables as the dimension of the aforementioned projection space.
	\end{abstract}
	
	
	
	\allowdisplaybreaks
	
	\section{Introduction}\label{intro}
	
	We consider a system of indistinguishable molecules of a chemical species $S$ which undergo 
	\begin{itemize}
		\item drift-less isotropic \emph{diffusion} in the interval $[0,1]$;
		\item \emph{creation} and \emph{mutual annihilation} chemical reactions
		\begin{equation*}
			\mbox{(I)}\quad \varnothing\xrightarrow{\lambda_c(x)}S\quad\quad\quad\mbox{(II)}\quad S+S\xrightarrow{\lambda_d(x,y)}\varnothing.
		\end{equation*}
	\end{itemize} 
Here, the function $[0,1]\ni x\mapsto\lambda_c(x)$ represents the stochastic rate function for reaction (I); it can be thought of being of the form $\lambda_c(x)=\gamma\pi_c(x)$ with $\gamma$ being a positive constant representing the probability per unit of time for a new particle to be created while $\pi_c$ is a probability density on $[0,1]$ which describes the random location for the birth of the new particle. Similarly, the function $[0,1]^2\ni (x,y)\mapsto\lambda_d(x,y)$ is the stochastic rate function for reaction (II) to occur between two particles located at $(x,y)$; for instance, when $\lambda_d$ is constant then the location of the two particles is not relevant for reaction (II) to take place; on the contrary, if $\lambda_d(x,y)=\delta(x-y)$ (here $\delta$ stands for the Dirac delta function with mass at zero) then reaction (II) occurs (with rate one) only for particles having the same location. \\
To analyze the probabilistic evolution of such system the authors in \cite{delRazo} (see also \cite{delRazo2} for a further discussion of the model) proposed a set of equations which describe how the number of molecules and their positions change with time. Namely, for $t\geq 0$, $n\geq 1$ and $A\in\mathcal{B}([0,1]^n)$ they set
	\begin{align*}
		\mathcal{N}(t)&:=\mbox{ number of molecules at time $t$},\\
		\rho_0(t)&:=\mathbb{P}(\mathcal{N}(t)=0),\\
		\int_{A}\rho_n(t,x_1,...,x_n)dx_1\cdot\cdot\cdot dx_n&:=\mathbb{P}\left(\{\mathcal{N}(t)=n\}\cap\{(X_1(t),...,X_n(t))\in A\}\right);
	\end{align*}
here $(X_1(t),...,X_n(t))$ is the vector collecting the positions at time $t$ of the $n$ particles constituting the system. Then, they write the following infinite system of equations: 
	\begin{align}\label{equation}
		\begin{cases}
			\begin{split}
				\partial_t\rho_n(t,x_1,...,x_n)=&\sum_{i=1}^n\partial^2_{x_i}\rho_n(t,x_1,...,x_n)\\
				&+\frac{(n+2)(n+1)}{2}\int_{[0,1]^2}\lambda_d(x,y)\rho_{n+2}(t,x_1,...,x_n,x,y)dxdy\\
				&-\sum_{i<j}\lambda_d(x_i,x_j)\cdot\rho_n(t,x_1,...,x_n)\\
				&+\frac{1}{n}\sum_{i=1}^n\lambda_c(x_i)\rho_{n-1}(t,x_1,...,x_{i-1},x_{i+1},...,x_n)\\
				&-\int_{[0,1]}\lambda_c(y)dy\cdot\rho_n(t,x_1,...,x_n),\quad\quad n\geq 0, t>0, (x_1,...,x_n)\in [0,1]^n,
			\end{split}
		\end{cases}
	\end{align}
	where we agree on assigning value zero to the three sums above when $n=0$. The term
	\begin{align*}
		\sum_{i=1}^n\partial^2_{x_i}\rho_n(t,x_1,...,x_n)
	\end{align*}
	in \eqref{equation} refers to spatial diffusion of the particles; the terms
	\begin{align*}
	\frac{(n+2)(n+1)}{2}\int_{[0,1]^2}\lambda_d(x,y)\rho_{n+2}(t,x_1,...,x_n,x,y)dxdy
	\end{align*}
and
	\begin{align*}
	\sum_{i<j}\lambda_d(x_i,x_j)\cdot\rho_n(t,x_1,...,x_n)
	\end{align*}
	formalize gain and loss, respectively, due to reaction (II), while
	\begin{align*}
		\frac{1}{n}\sum_{i=1}^n\lambda_c(x_i)\rho_{n-1}(t,x_1,...,x_{i-1},x_{i+1},...,x_n)
	\end{align*}
and
\begin{align*}	
		-\int_{[0,1]}\lambda_c(y)dy\cdot\rho_n(t,x_1,...,x_n)
	\end{align*}
	represent gain and loss, respectively, associated to reaction (I). System \eqref{equation} is combined with initial and Neumann boundary conditions 
	\begin{align}\label{initial}
		\begin{cases}
			\begin{split}
				\rho_0(0)&=1;\\
				\rho_n(0,x_1,...,x_n)&=0,\quad n\geq 1, (x_1,...,x_n)\in [0,1]^n;\\
				\partial_{\nu}\rho_n(t,x_1,...,x_n)&=0,\quad n\geq 1, t\geq 0, (x_1,...,x_n)\in\partial [0,1]^n.
			\end{split}
		\end{cases}
	\end{align}
	The initial condition (first two equations in \eqref{initial}) states that there are no molecules in the system at time zero while the Neumann condition prevents flux through the boundary of $[0,1]$, thus forcing the diffusion of the molecules inside $[0,1]$. The symbol $\partial_{\nu}$ in \eqref{initial} stands for the directional derivative along the outer normal vector at the boundary of $[0,1]^n$.

\subsection{Literature review}
	
The dynamics of biochemical processes in living cells are commonly understood as an interplay between the spatial transport (diffusion) of molecules and their chemical kinetics (reaction), both of which are inherently stochastic at the molecular scale. In the case of systems with small molecule numbers in spatially well-mixed settings, the diffusion is averaged out and the probabilistic dynamics are governed by the well-known chemical master equation (CME) \cite{gillespie1977exact,qian2010chemical, qian2021stochastic}. The CME can be seldom solved analytically \cite{jahnke2007solving}. However, solving a few simple cases analytically can bring valuable insight to the solutions of more complex cases. Alternatively, one can solve it by integrating stochastic trajectories with the Gillespie or tau-leap algorithms \cite{anderson2015stochastic, gillespie1977exact}, by approximation methods \cite{deuflhard2008adaptive,engblom2009spectral,munsky2006finite,schnoerr2017approximation} or even by deep learning approaches \cite{gupta2021deepcme, jiang2021neural}.

In the case of spatially inhomogeneous systems, where diffusion is not averaged out, one would expect to obtain a similar master equation. However, obtaining such an equation is plagued with mathematical difficulties, and although it was hinted in previous work \cite{doi1976second} and formulated for some specific systems \cite{schweitzer2003brownian}, it was not until recently that this was formalized into the so-called chemical diffusion master equation (CDME) \cite{delRazo,delRazo2}. The CDME changes a few paradigms that have not yet been explored thoroughly in stochastic chemical kinetics models. It combines continuous and discrete degrees of freedom, and it models reaction and diffusion as a joint stochastic process. It consists of an infinite sorted family of Fokker-Planck equations, where each level of the sorted family corresponds to a certain number of particles/molecules. The equations at each level describe the spatial diffusion of the corresponding set of particles, and they are coupled to each other via reaction operators, which change the number of particles in the system. The CDME is the theoretical backbone of reaction-diffusion processes, and thus, it is fundamental to model and understand biochemical processes in living cells, as well as to develop multiscale numerical methods \cite{del2018grand,flegg2012two,kostre2021coupling,smith2018spatially} and hybrid algorithms \cite{chen2014brownian,dibak2018msm,del2021multiscale}. The stochastic trajectories of the CDME can be often integrated using particle--based reaction--diffusion simulations \cite{ andrews2017smoldyn, hoffmann2019readdy}. 

The problem of finding analytical solutions to some CDMEs has been recently addressed in the papers \cite{Lanconelli2023} and \cite{bdCDME}. In \cite{Lanconelli2023} the author proposed an infinite dimensional version of the classical moment generating function method, which is commonly utilized to find analytical solution to some CMEs \cite{Lecca},\cite{EC},\cite{McQuarrie} (see also \cite{Zhang2005}). In \cite{Lanconelli2023} the method is employed to solve a CDME of birth-death type; this approach has been further explored in \cite{bdCDME} and an explicit solution for the general birth-death CDME is presented.

\subsection{Aim of the paper and organization of the material}

The aim of this paper is twofold. First, in Section 2 we propose a rigorous probabilistic derivation of equation \eqref{equation}. The idea is to mimic the procedure for obtaining the backward Kolmogorov equation of a birth-death continuous time Markov chain: see for instance \cite{KT} for details. To this aim we will fix a set of transition probabilities (see Assumption \ref{assumption transitions} below) and derive through a limit argument the desired equation. We remark that our approach can be readily generalized to include higher order chemical reactions and more complex descriptions of the diffusive motion of the particles (i.e. anisotropic diffusion with drift). Then, in Section 3 we employ the general method proposed in \cite{Lanconelli2023} to analytically solve equation \eqref{equation}-\eqref{initial}. This requires the use of Gaussian Malliavin calculus's techniques (summarized in the Appendix below) and provides a link between the solution to \eqref{equation}-\eqref{initial} and the solution of a single fourth order linear PDE which describes certain finite dimensional projections of the solution to the original problem. At the end, some comments on the Guassian features introduced in our problem by the proposed approach are also discussed. 
	
\section{Probabilistic derivation of equation \eqref{equation} as a backward Kolmogorov equation}\label{Section 2}
	
	In this section we present a probabilistic derivation of equation \eqref{equation}. We recall that the particles of the system under investigation are subject to the chemical reactions
	\begin{align}\label{CR}
	\mbox{(I)}\quad \varnothing\xrightarrow{\lambda_c(x)}S\quad\quad\quad\mbox{(II)}\quad S+S\xrightarrow{\lambda_d(x,y)}\varnothing,
	\end{align}
	and diffuse in space as independent Brownian motions between successive reactions.\\
	In the sequel we will be dealing with probabilities of the form $\mathbb{P}(\mathcal{N}(t)=n,X(t)\in A)$: this represents the probability that the system at time $t$ is made of $n$ many particles and that such particles  are located in the region $A$. We are not going to use an extra index in $X(t)$ to stress that it is an $n$-dimensional vector; this vector will always come with an event of the type $\{\mathcal{N}(t)=n\}$ and hence the number of components of $X(t)$ will be uniquely determined.
	
	We begin with a couple of technical assumptions.  
	\begin{assumption}\label{density}
	For any $n\geq 1$ and $t>0$ there exists a symmetric function $\rho_n(t,x_1,...,x_n)$ such that
	\begin{align*}
		\mathbb{P}(\mathcal{N}(t)=n,X(t)\in A)=\int_A\rho_n(t,x_1,...,x_n)dx_1\cdot\cdot\cdot dx_n;
	\end{align*}
we also set 
\begin{align*}
\rho_0(t):=\mathbb{P}(\mathcal{N}(t)=0).
\end{align*} 
\end{assumption}
Notice that the symmetry of the functions $\rho_n$ models the indistinguishability of the particles in the system; moreover, by construction the sequence $\{\rho_n\}_{n\geq 0}$ fulfils the constraint
\begin{align}\label{constraint}
	\sum_{n\geq 0}\int_{[0,1]^n}\rho_n(t,x_1,...,x_n)dx_1\cdot\cdot\cdot dx_n=1.
\end{align}

\begin{assumption}\label{lambda}
	The functions $\lambda_c$ and $\lambda_d$ appearing in \eqref{CR} are non negative, bounded and continuous. Moreover, $\lambda_d(x,y)=\lambda_d(y,x)$ for all $x,y\in [0,1]$.
\end{assumption}

We are now ready to describe the probabilistic structure to be imposed on our system for the formal derivation of equation \eqref{equation}.  
	
	\begin{assumption}\label{assumption transitions}
		The system under investigation possesses the following properties:
		\begin{itemize}
		\item \emph{Diffusion of particles}: in absence of chemical reactions, particles diffuse in $[0,1]$ like independent Brownian motions with variance $2t$ and reflecting boundary conditions; more precisely, the transition density $\{p_t(x|y)\}_{t\geq 0,x,y\in [0,1]^n}$ for the motion of $n$ many particles solves 
		\begin{align*}
		\begin{cases}
		\partial_tu(t,x|y)=\sum_{j=1}^n\partial^2_{x_i}u(t,x|y),& t>0,x,y\in [0,1]^n;\\
		u(0,x|y)=\delta_y(x),& x,y\in [0,1]^n;\\
		\partial_{\nu}u(t,x|y)=0,& t\geq 0, x\in\partial[0,1]^n,y\in [0,1]^n.
		\end{cases}
		\end{align*} 
		\item \emph{Reaction (I) + diffusion of particles}: for any $n\geq 1$, $A\in\mathcal{B}([0,1]^n)$ and $t,h>0$ we have
		\begin{align}\label{C}
			&\mathbb{P}(\mathcal{N}(t+h)=n,X(t+h)\in A|\mathcal{N}(t)=n-1,X(t)=y)\nonumber\\
			&\quad\quad=h\frac{1}{n}\sum_{j=1}^{n}\int_A\int_{[0,1]}\mathtt{p}_h(x|y\cup_j z)\lambda_c(z)dzdx+\mathcal{O}(h^2),
		\end{align}
		with $y\cup_j z:=(y_1,...,y_{j-1},z,y_{j+1},...,y_{n})\in[0,1]^n$. To explain the contribution of each single term of the identity above, we imagine to split the function $\lambda_c$ as $\gamma\cdot\pi_d$ where $\gamma:=\int_{[0,1]}\lambda_c(z)dz$ while $\pi_c$ is a probability density function supported on $[0,1]$. The chemical reaction (I) adds a new particle, here denoted with $z$, to the system: the rate at which this happens is $\gamma$ while the location for the birth of the particle is distributed according to $\pi_d$. Once the creation takes place, the $n$ particles of the system diffuse in space and this is formalized through the expression $\int_A\int_{[0,1]}\mathtt{p}_h(x|y\cup_j z)\lambda_c(z)dzdx$. Then, to make particles indistinguishable we symmetrize over the possible positions of $z$ in the vector $y\cup_j z$.
		\item \emph{Reaction (II) + diffusion of particles}: for any $n\geq 0$, $A\in\mathcal{B}([0,1]^n)$ and $t,h>0$ we have
		\begin{align}\label{MA}
			&\mathbb{P}(\mathcal{N}(t+h)=n,X(t+h)\in A|\mathcal{N}(t)=n+2,X(t)=y)\nonumber\\
			&\quad\quad=h\sum_{j<k}\lambda_d(y_j,y_k)\int_A\mathtt{p}_h(x|\hat{y}_{j,k})dx+\mathcal{O}(h^2),
		\end{align}
		with $\hat{y}_{j,k}:=(y_1,...,y_{j-1},y_{j+1},...,y_{k-1},y_{k+1},...,y_n)\in [0,1]^n$. The chemical reaction (II) removes two particles from the system while the others diffuse: this is the contribution of $\int_A\mathtt{p}_h(x|\hat{y}_{j,k})dx$ where the particles labelled $j$ and $k$ are those undergoing the chemical reaction. We have then to mediate this over all the possible couples of particles in the system: the weights of this average, represented by the sum above, are provided by $\lambda_d$ which measures the likelihood for two particles to react depending on their locations. We also mention that for $n=0$ the right hand side of \eqref{MA} simplifies to $h\lambda_d(x_1,x_2)+\mathcal{O}(h^2)$. 
		\item \emph{No reactions + diffusion of particles}: for any $n\geq 1$, $A\in\mathcal{B}([0,1]^n)$ and $t,h>0$ we have
		\begin{align}\label{NR}
		&\mathbb{P}(\mathcal{N}(t+h)=n,X(t+h)\in A|\mathcal{N}(t)=n,X(t)=y)\nonumber\\
		&\quad\quad=\left(1-h\int_{[0,1]}\lambda_c(z)dz-h\sum_{j< k}\lambda_d(y_j,y_k)\right)\int_A\mathtt{p}_h(x|y)dx+\mathcal{O}(h^2).
		\end{align}
		The term inside parenthesis reflects the probability of no reaction happening while the integral formalizes the diffusion of particles.
		\item \emph{Multiple reactions}: for any $n\geq 1$, $A\in\mathcal{B}([0,1]^n)$ and $t,h>0$ we have
		\begin{align}\label{MR}
		\mathbb{P}(\mathcal{N}(t+h)=n,X(t+h)\in A|\mathcal{N}(t)=k,X(t)=y)=\mathcal{O}(h^2),
		\end{align}
		whenever $k\notin\{n-1,n,n+2\}$.
	\end{itemize}
	\end{assumption}

We now show how to use Assumption \ref{density}-\ref{lambda}-\ref{assumption transitions} to get the CDME \eqref{equation}. Let $n\geq 1$ and $A\in\mathcal{B}([0,1]^n)$; then, according to the law of total probability we can write
	\begin{align}\label{1}
		\int_A\rho_n(t+h,x)dx\nonumber=&\mathbb{P}(\mathcal{N}(t+h)=n,X(t+h)\in A)\nonumber\\
		=&\sum_{k\geq 0}\int_{[0,1]^k}\mathbb{P}(\mathcal{N}(t+h)=n,X(t+h)\in A|\mathcal{N}(t)=k,X(t)=y)\mathbb{P}(\mathcal{N}(t)=k,X(t)\in dy)\nonumber\\
		=&\sum_{k\geq 0}\int_{[0,1]^k}\mathbb{P}(\mathcal{N}(t+h)=n,X(t+h)\in A|\mathcal{N}(t)=k,X(t)=y)\rho_k(t,y)dy.
	\end{align}
	Notice that for $k=0$ the corresponding term in the sum above should be interpreted as
	\begin{align*}
		\mathbb{P}(\mathcal{N}(t+h)=n,X(t+h)\in A|\mathcal{N}(t)=0)\rho_0(t).
	\end{align*}
Now, in view of Assumption \ref{assumption transitions} the only transitions of order one in $h$ are those with $k=n+2$, $k=n-1$ and $k=n$ while the others are of order at least two; therefore, we can rewrite \eqref{1} as
\begin{align}\label{2}
	\int_A\rho_n(t+h,x)dx=&\sum_{k\geq 0}\int_{[0,1]^k}\mathbb{P}(\mathcal{N}(t+h)=n,X(t+h)\in A|\mathcal{N}(t)=k,X(t)=y)\rho_k(t,y)dy\nonumber\\
	=&\int_{[0,1]^{n+2}}\mathbb{P}(\mathcal{N}(t+h)=n,X(t+h)\in A|\mathcal{N}(t)=n+2,X(t)=y)\rho_{n+2}(t,y)dy\nonumber\\
	&+\int_{[0,1]^{n-1}}\mathbb{P}(\mathcal{N}(t+h)=n,X(t+h)\in A|\mathcal{N}(t)=n-1,X(t)=y)\rho_{n-1}(t,y)dy\nonumber\\
	&+\int_{[0,1]^n}\mathbb{P}(\mathcal{N}(t+h)=n,X(t+h)\in A|\mathcal{N}(t)=n,X(t)=y)\rho_n(t,y)dy\nonumber\\
	&+\mathcal{O}(h^2)\nonumber\\
	=&\int_{[0,1]^{n+2}}\left(\int_Ah\sum_{j<k}\lambda_d(y_j,y_k)\mathtt{p}_h(x|\hat{y}_{j,k})dx\right)\rho_{n+2}(t,y)dy\nonumber\\
	&+\int_{[0,1]^{n-1}}\left(\int_A\frac{h}{n}\sum_{j=1}^{n}\int_{[0,1]}\lambda_c(z)\mathtt{p}_h(x|y\cup_j z)dzdx\right)\rho_{n-1}(t,y)dy\nonumber\\
	&+\int_{[0,1]^{n}}\left(\int_A\left(1-h\int_{[0,1]}\lambda_c(z)dz-h\sum_{j<k}\lambda_d(y_j,y_k)\right)\mathtt{p}_h(x|y)dx\right)\rho_n(t,y)dy\nonumber\\
	&+\mathcal{O}(h^2)\nonumber\\
	=&h\int_A\left(\sum_{j<k}\int_{[0,1]^{n+2}}\lambda_d(y_j,y_k)\mathtt{p}_h(x|\hat{y}_{j,k})\rho_{n+2}(t,y)dy\right)dx\nonumber\\
	&+h\int_A\left(\frac{1}{n}\sum_{j=1}^{n}\int_{[0,1]^{n-1}}\int_{[0,1]}\lambda_c(z)\mathtt{p}_h(x|y\cup_j z)\rho_{n-1}(t,y)dzdy\right)dx\nonumber\\
	&+\int_A\left(\int_{[0,1]^{n}}\left(1-h\int_{[0,1]}\lambda_c(z)dz-h\sum_{j<k}\lambda_d(y_j,y_k)\right)\mathtt{p}_h(x|y)\rho_n(t,y)dy\right)dx\nonumber\\
	&+\mathcal{O}(h^2)\nonumber\\
	=&h\int_A\left(\sum_{j<k}\int_{[0,1]^{2}}\lambda_d(y_j,y_k)\left(\int_{[0,1]^{n}}\mathtt{p}_h(x|\hat{y}_{j,k})\rho_{n+2}(t,y)d\hat{y}_{j,k}\right)dy_jdy_k\right)dx\nonumber\\
	&+h\int_A\left(\frac{1}{n}\sum_{j=1}^{n}\int_{[0,1]^{n-1}}\int_{[0,1]}\lambda_c(z)\mathtt{p}_h(x|y\cup_j z)\rho_{n-1}(t,y)dzdy\right)dx\\
	&+\int_A\left(\int_{[0,1]^{n}}\left(1-h\int_{[0,1]}\lambda_c(z)dz-h\sum_{j< k}\lambda_d(y_j,y_k)\right)\mathtt{p}_h(x|y)\rho_n(t,y)dy\right)dx\nonumber\\
	&+\mathcal{O}(h^2)\nonumber.			
\end{align}
	To ease the notation we now introduce the following:
	\begin{align*}
		\mathtt{T}_hf(x):=\int_{[0,1]^n}\mathtt{p}_h(x|y)f(y)dy,\quad f\in C_0([0,1]^n),
	\end{align*}
	and recall that for suitably regular $f$ we have
	\begin{align}\label{delta}
	\lim_{h\to 0}\mathtt{T}_hf(x)=f(x),\quad x\in[0,1]^n,
	\end{align}
	and
	\begin{align}\label{delta 2}
	\lim_{h\to 0}\frac{\mathtt{T}_hf(x)-f(x)}{h}=\sum_{i=1}^n\partial_{x_i}^2f(x)\quad x\in[0,1]^n;
	\end{align}
	we refer to \cite{Kallenberg} for a precise formulation of those statements. With this notation at hand we can rewrite \eqref{2} as 
\begin{align*}
\int_A\rho_n(t+h,x)dx=&h\int_A\left(\sum_{j<k}\int_{[0,1]^{2}}\lambda_d(y_j,y_k)(\mathtt{T}_h\rho_{n+2}(t,\cdot,y_j,y_k))(x)dy_jdy_k\right)dx\\
&+h\int_A\frac{1}{n}\sum_{j=1}^{n}\mathtt{T}_h(\lambda_c\otimes_j\rho_{n-1}(t,\cdot))(x)dx\\
&+\int_A(\mathtt{T}_h\rho_n(t,\cdot))(x)dx-h\int_{[0,1]}\lambda_c(z)dz\int_A(\mathtt{T}_h\rho_n(t,\cdot))(x)dx\\
&-h\sum_{j< k}(\mathtt{T}_h\lambda_d(\cdot_j,\cdot_k)\rho_n(t,\cdot))(x)dx+\mathcal{O}(h^2).
\end{align*}	
We now subtract the quantity $\int_A\rho_n(t,x)dx$ from both sides of the last equality, divide by $h$ and take the limit as $h$ tends to zero. This gives
\begin{align*}
\int_A\partial_t\rho_n(t,x)dx=&\lim_{h\to 0}\int_A\frac{\rho_n(t+h,x)-\rho_n(t,x)}{h}dx\\
=&\lim_{h\to 0}\int_A\left(\sum_{j<k}\int_{[0,1]^{2}}\lambda_d(y_j,y_k)(\mathtt{T}_h\rho_{n+2}(t,\cdot,y_j,y_k))(x)dy_jdy_k\right)dx\\
&+\lim_{h\to 0}\int_A\frac{1}{n}\sum_{j=1}^{n}\mathtt{T}_h(\lambda_c\otimes_j\rho_{n-1}(t,\cdot))(x)dx\\
&+\lim_{h\to 0}\int_A\frac{(\mathtt{T}_h\rho_n(t,\cdot))(x)-\rho_n(t,x)}{h}dx-\lim_{h\to 0}\int_{[0,1]}\lambda_c(z)dz\int_A(\mathtt{T}_h\rho_n(t,\cdot))(x)dx\\
&-\lim_{h\to 0}\int_A\sum_{j< k}(\mathtt{T}_h\lambda_d(\cdot_j,\cdot_k)\rho_n(t,\cdot))(x)dx.
\end{align*} 	
Now, using \eqref{delta} we get
\begin{align*}
	&\lim_{h\to 0}\int_A\left(\sum_{j<k}\int_{[0,1]^{2}}\lambda_d(y_j,y_k)(\mathtt{T}_h\rho_{n+2}(t,\cdot,y_j,y_k))(x)dy_jdy_k\right)dx\\
	&\quad\quad=\int_A\left(\sum_{j<k}\int_{[0,1]^{2}}\lambda_d(y_j,y_k)\rho_{n+2}(t,x,y_j,y_k)dy_jdy_k\right)dx\\
	&\quad\quad=\int_A\left(\frac{(n+2)(n+1)}{2}\int_{[0,1]^{2}}\lambda_d(y_j,y_k)\rho_{n+2}(t,x,y_j,y_k)dy_jdy_k\right)dx,
\end{align*}	
and
\begin{align*}
	\lim_{h\to 0}\int_A\frac{1}{n}\sum_{j=1}^{n}\mathtt{T}_h(\lambda_c\otimes_j\rho_{n-1}(t,\cdot))(x)dx=\int_A\left(\frac{1}{n}\sum_{j=1}^{n}\lambda_c(x_j)\rho_{n-1}(t,x_1,...,x_{j-1},x_{j+1},...,x_n)\right)dx.
\end{align*}
Moreover, formula \eqref{delta 2} yields
\begin{align*}
\lim_{h\to 0}\int_A\frac{(\mathtt{T}_h\rho_n(t,\cdot))(x)-\rho_n(t,x)}{h}dx=\int_A\left(\sum_{i=1}^n\partial_{x_i}^2\rho_n(t,x)\right)dx,
\end{align*}
while formula \eqref{delta} gives 
\begin{align*}
\lim_{h\to 0}\int_{[0,1]}\lambda_c(z)dz\int_A(\mathtt{T}_h\rho_n(t,\cdot))(x)dx=\int_{[0,1]}\lambda_c(z)dz\int_A\rho_n(t,x)dx.
\end{align*}	
and
\begin{align*}
	\lim_{h\to 0}\int_A\sum_{j< k}(\mathtt{T}_h\lambda_d(\cdot_j,\cdot_k)\rho_n(t,\cdot))(x)dx=\int_A\sum_{j< k}\lambda_d(x_j,x_k)\rho_n(t,x)dx.
\end{align*}
If we combine all the preceding equalities we can conclude that
\begin{align*}
	\int_A\partial_t\rho_n(t,x)dx=&\int_A\left(\frac{(n+2)(n+1)}{2}\int_{[0,1]^{2}}\lambda_d(y_j,y_k)\rho_{n+2}(t,x,y_j,y_k)dy_jdy_k\right)dx\\
	&+\int_A\left(\frac{1}{n}\sum_{j=1}^{n}\lambda_c(x_j)\rho_{n-1}(t,x_1,...,x_{j-1},x_{j+1},...,x_n)\right)dx\\
	&+\int_A\left(\sum_{i=1}^n\partial_{x_i}^2\rho_n(t,x)\right)dx\\
	&-\int_{[0,1]}\lambda_c(z)dz\int_A\rho_n(t,x)dx\\
	&-\int_A\sum_{j< k}\lambda_d(x_j,x_k)\rho_n(t,x)dx.
\end{align*} 
Since $A\in\mathcal{B}([0,1]^n)$ is arbitrary, this is equivalent to \eqref{equation}.
	
\section{Analysis of equation \eqref{equation}-\eqref{initial} through an infinite dimensional generating function method}
	
	In this section we employ the general method proposed in \cite{Lanconelli2023} to solve analytically the CDME \eqref{equation}-\eqref{initial}. We mention that this method has lead to an explicit representation for the solution to the general birth-death CDME \cite{bdCDME}. For the application of that approach in the current framework we need to impose the following technical condition.
	
	\begin{assumption}\label{constant}
		The function $\lambda_d:[0,1]^2\to[0,+\infty[$ is constant; this means that equation \eqref{equation} simplifies to
\begin{align}\label{equation2}
	\begin{cases}
		\begin{split}
			\partial_t\rho_n(t,x_1,...,x_n)=&\sum_{i=1}^n\partial^2_{x_i}\rho_n(t,x_1,...,x_n)\\
			&+\lambda_d\frac{(n+2)(n+1)}{2}\int_{[0,1]^2}\rho_{n+2}(t,x_1,...,x_n,x,y)dxdy\\
			&-\lambda_d\frac{n(n-1)}{2}\rho_n(t,x_1,...,x_n)\\
			&+\frac{1}{n}\sum_{i=1}^n\lambda_c(x_i)\rho_{n-1}(t,x_1,...,x_{i-1},x_{i+1},...,x_n)\\
			&-\gamma\rho_n(t,x_1,...,x_n),\quad\quad\quad n\geq 0, t>0, (x_1,...,x_n)\in [0,1]^n,
		\end{split}
	\end{cases}
\end{align}
\end{assumption}

The need for such assumption is related to some technical features of the method used to solve \eqref{equation2}. We refer to Remark \ref{commutation bound} below for details. 

\begin{remark}
	As pointed out in \cite{delRazo}, if we integrate equation \eqref{equation2}-\eqref{initial} with respect to all degrees of freedom, then this reduces to the classical CME for the reactions 
	\begin{align*}
		\mbox{(I)}\quad \varnothing\xrightarrow{\gamma}S\quad\quad\quad\mbox{(II)}\quad S+S\xrightarrow{\lambda_d}\varnothing.
	\end{align*}
	In fact, identity 
	\begin{align*}
	\mathbb{P}(\mathcal{N}(t)=n)=\int_{[0,1]^n}\rho_n(t,x_1,...,x_n)dx_1\cdot\cdot\cdot dx_n,
	\end{align*}
(compare with Assumption \ref{density}) together with the boundary conditions in \eqref{initial} yield
	\begin{align}\label{equation book}
		\partial_t\mathbb{P}(\mathcal{N}(t)=n)=&\lambda_d\frac{(n+2)(n+1)}{2}\mathbb{P}(\mathcal{N}(t)=n+2)-\lambda_d\frac{n(n-1)}{2}\mathbb{P}(\mathcal{N}(t)=n)\nonumber\\
		&+\gamma \mathbb{P}(\mathcal{N}(t)=n-1)-\gamma\mathbb{P}(\mathcal{N}(t)=n),
	\end{align} 
which is indeed the desired CME (see equation (1.28) in \cite{EC}: here, the authors employ the generating function method for solving equation \eqref{equation book}).
\end{remark}

The scheme for solving \eqref{equation2}, as presented in \cite{Lanconelli2023}, is made of several steps that we now discuss in the following preparatory results. Before doing that, we introduce the notation $-\mathcal{A}:=\partial_x^2$ and we recall that in the Appendix below one can find a quick review of the Malliavin calculus's tools utilized in the sequel.

	\begin{lemma}\label{infinite}
	If $\{\rho_n\}_{n\geq 0}$ is a classical solution to equation \eqref{equation2}-\eqref{initial}, then 
	\begin{align}\label{Phi}
			\Phi(t):=\sum_{n\geq 0}I_n(\rho_n(t,\cdot))
	\end{align}
solves 
\begin{align}\label{SDE}
	\begin{cases}
		\partial_t\Phi=d\Gamma(-\mathcal{A})\Phi+\frac{\lambda_d}{2}D^2_{1}\Phi-\frac{\lambda_d}{2}\mathtt{N}(\mathtt{N}-\mathtt{I})\Phi+D^{\star}_{\lambda_c}\Phi-\gamma\Phi;\\
		\Phi(0)=1,
	\end{cases}
\end{align}
in $\mathtt{F}^{\star}$.
	\end{lemma}

\begin{proof}
Let $\{\rho_n\}_{n\geq 0}$ be a classical solution to equation \eqref{equation2}-\eqref{initial}; this means in particular that $\rho_0\in C^1([0,+\infty[)$ and $\rho_n\in C^{1,2}([0,+\infty[\times [0,1]^n)$ for all $n\geq 1$. Recall also that according to Assumption \ref{density}, for any $n\geq 2$ and $t\geq 0$ the function $\rho_n(t,\cdot)$ is symmetric in its arguments. This allows us to consider the multiple It\^o integrals $I_n(\rho_n(t,\cdot))$ and to interchange the partial derivative $\partial_t$ with the iterated integrals. Furthermore, employing the operators $d\Gamma(-\mathcal{A})$, $D_1$, $\mathtt{N}$ and $D_{\lambda_c}^{\star}$, whose definitions and properties can be found in the Appendix below, and equation \eqref{equation2}, we can write for all $n\geq 1$ and $t\geq 0$ that
\begin{align*}
			\partial_tI_n(\rho_n(t,\cdot))=&I_n(\partial_t\rho_n(t,\cdot))\\
			=&d\Gamma(-\mathcal{A})I_n(\rho_n(t,\cdot))+\frac{\lambda_d}{2}D_1^2I_{n+2}(\rho_{n+2}(t,\cdot))-\frac{\lambda_d}{2}\mathtt{N}(\mathtt{N}-\mathtt{I})I_n(\rho_n(t,\cdot))\\
			&+D_{\lambda_c}^{\star}I_{n-1}(\rho_{n-1}(t,\cdot))-\gamma I_n(\rho_n(t,\cdot)).
\end{align*} 
If we now sum over $n\geq 0$ and recall that $\rho_{-1}\equiv 0$ while $D_1^2$ maps to zero any multiple It\^o integral of order less than two, we obtain equation \eqref{SDE} for the stochastic process defined in \eqref{Phi}.
\end{proof}

\begin{remark}
	It is worth to point out that condition \eqref{constraint} is already encoded in equation \eqref{SDE}. In fact, if $\{\Phi(t)\}_{t\geq 0}$ solves \eqref{SDE}, then
	\begin{align*}
		\partial_t\langle\langle \Phi,\mathcal{E}(1)\rangle\rangle=&\langle\langle \partial_t\Phi,\mathcal{E}(1)\rangle\rangle\\
		=&\langle\langle d\Gamma(-\mathcal{A})\Phi,\mathcal{E}(1)\rangle\rangle\\
		&+\frac{\lambda_d}{2}\langle\langle D^2_{1}\Phi,\mathcal{E}(1)\rangle\rangle-\frac{\lambda_d}{2}\langle\langle\mathtt{N}(\mathtt{N}-\mathtt{I})\Phi,\mathcal{E}(1)\rangle\rangle\\
		&+\langle\langle D^{\star}_{\lambda_c}\Phi,\mathcal{E}(1)\rangle\rangle-\gamma\langle\langle\Phi,\mathcal{E}(1)\rangle\rangle.
	\end{align*}
Now,
\begin{align*}
	\langle\langle d\Gamma(-\mathcal{A})\Phi,\mathcal{E}(1)\rangle\rangle&=\langle\langle\Phi,d\Gamma(-\mathcal{A})\mathcal{E}(1)\rangle\rangle=\langle\langle\Phi,\mathcal{E}(1)\diamond\delta(-\mathcal{A}1)\rangle\rangle=0,
\end{align*}
while
\begin{align*}
	\langle\langle D^2_{1}\Phi,\mathcal{E}(1)\rangle\rangle-\langle\langle\mathtt{N}(\mathtt{N}-\mathtt{I})\Phi,\mathcal{E}(1)\rangle\rangle=&\langle\langle \Phi,D^{\star}_{1}D^{\star}_{1}\mathcal{E}(1)\rangle\rangle-\langle\langle\Phi,\mathtt{N}(\mathtt{N}-\mathtt{I})\mathcal{E}(1)\rangle\rangle\\
	=&\langle\langle \Phi,\mathcal{E}(1)\diamond\delta(1)\diamond\delta(1)\rangle\rangle-\langle\langle\Phi,\mathtt{N}^2\mathcal{E}(1)-\mathtt{N}\mathcal{E}(1)\rangle\rangle\\
	=&\langle\langle \Phi,\mathcal{E}(1)\diamond\delta(1)\diamond\delta(1)\rangle\rangle-\langle\langle\Phi,\mathcal{E}(1)\diamond\delta(1)\diamond\delta(1)\rangle\rangle\\
	=&0.
\end{align*}
Here, in the last equality we utilized the identities
\begin{align*}
	\mathtt{N}\mathcal{E}(1)=\mathcal{E}(1)\diamond\delta(1)\quad\mbox{ and }\quad\mathtt{N}^2\mathcal{E}(1)=\mathcal{E}(1)\diamond\delta(1)\diamond\delta(1)+\mathcal{E}(1)\diamond\delta(1).
\end{align*}
Lastly,
\begin{align*}
	\langle\langle D^{\star}_{\lambda_c}\Phi,\mathcal{E}(1)\rangle\rangle-\gamma\langle\langle\Phi,\mathcal{E}(1)\rangle\rangle=&\langle\langle \Phi,D_{\lambda_c}\mathcal{E}(1)\rangle\rangle-\gamma\langle\langle\Phi,\mathcal{E}(1)\rangle\rangle\\
	=&\int_{[0,1]}\lambda_c(x)dx\cdot\langle\langle \Phi,\mathcal{E}(1)\rangle\rangle-\gamma\langle\langle\Phi,\mathcal{E}(1)\rangle\rangle\\
	=&0.
\end{align*}
This proves that $\partial_t\langle\langle \Phi(t),\mathcal{E}(1)\rangle\rangle=0$;  since $\Phi(0)=1$ (recall the initial condition in \eqref{SDE} which in turn follows from \eqref{initial}) we deduce that $\langle\langle \Phi(0),\mathcal{E}(1)\rangle\rangle=1$ and hence 
\begin{align}\label{dd}
\langle\langle \Phi(t),\mathcal{E}(1)\rangle\rangle=1,\quad\mbox{ for all $t\geq 0$}. 
\end{align}
On the other hand, by definition of dual pairing we can write
\begin{align*}
	\langle\langle \Phi(t),\mathcal{E}(1)\rangle\rangle=\sum_{n\geq 0}\int_{[0,1]^n}\rho_n(t,x_1,...,x_n)dx_1\cdot\cdot\cdot dx_n,
\end{align*}
which together with \eqref{dd} implies \eqref{constraint}. Notice that a similar calculation can be carried also when $\lambda_d$ is not constant; in this case the operator $\Phi\mapsto\mathtt{N}(\mathtt{N}-\mathtt{I})\Phi$ should be replaced with $\Phi\mapsto\delta^2(\lambda(\cdot,\cdot)D^2\Phi)$. This shows that condition \eqref{constraint} is part of equation \eqref{equation2} also in the absence of Assumption \ref{constant}.  
\end{remark}

The usefulness of transforming the CDME \eqref{equation2}-\eqref{initial} into the abstract problem \eqref{SDE} becomes apparent when we consider suitable finite dimensional projections of the stochastic process $\{\Phi(t)\}_{t\geq 0}$. To this aim, we set for $k\geq 1$
\begin{align}\label{xi}
\xi_k(x)=\sqrt{2}\cos((k-1)\pi x), x\in [0,1]\quad\mbox{ and }\quad\alpha_k=(k-1)^2\pi^2
\end{align}
to be the eigenfunctions with corresponding eigenvalues of the differential operator $-\mathcal{A}=\partial_x^2$ with homogeneous Neumann boundary conditions (as prescribed in \eqref{initial}). We also write $\Pi_N$ for the orthogonal projection onto the linear span of $\{\xi_1,...,\xi_N\}$.

\begin{lemma}\label{projection}
	If $\{\Phi(t)\}_{t\geq 0}$ solves \eqref{SDE} in $\mathtt{F}^{\star}$, then
	\begin{align}\label{Phi_N}
	\Phi_N(t):=\Gamma(\Pi_N)\Phi(t),\quad t\geq 0 
	\end{align}
	solves
	\begin{align}\label{SDE 2}
		\begin{cases}
			\partial_t\Phi_N=d\Gamma(-\mathcal{A})\Phi_N+\frac{\lambda_d}{2}D^2_{1}\Phi_N-\frac{\lambda_d}{2}\mathtt{N}(\mathtt{N}-\mathtt{I})\Phi_N+D^{\star}_{\Pi_N\lambda_c}\Phi_N-\gamma\Phi_N;\\
			\Phi_N(0)=1,
		\end{cases}
	\end{align}
	in $\mathtt{F}^{\star}$.
\end{lemma}

\begin{proof}
Using \eqref{Phi_N} and \eqref{SDE} we can write
\begin{align}\label{proof projection}
	\partial_t\Phi_N&=\partial_t\Gamma(\Pi_N)\Phi=\Gamma(\Pi_N)\partial_t\Phi\nonumber\\
	&=\Gamma(\Pi_N)\left(d\Gamma(-\mathcal{A})\Phi+\frac{\lambda_d}{2}D^2_{1}\Phi-\frac{\lambda_d}{2}\mathtt{N}(\mathtt{N}-\mathtt{I})\Phi+D^{\star}_{\lambda_c}\Phi-\gamma\Phi\right).
\end{align}
The proof consists in showing that our assumptions allow for the commutation between the operator $\Gamma(\Pi_N)$ and each of the following: $d\Gamma(-\mathcal{A})$, $D^2_{1}$, $\mathtt{N}(\mathtt{N}-\mathtt{I})$ and $D^{\star}_{\lambda_d}$. Let us start with the commutation between $\Gamma(\Pi_N)$ and $d\Gamma(-\mathcal{A})$: for any smooth $h\in L^2([0,1])$ we have
\begin{align*}
	\langle\langle\Gamma(\Pi_N)d\Gamma(-\mathcal{A})\Phi,\mathcal{E}(h)\rangle\rangle=&	\langle\langle d\Gamma(-\mathcal{A})\Phi,\mathcal{E}(\Pi_N h)\rangle\rangle\\
	=&\langle\langle\Phi,d\Gamma(-\mathcal{A})\mathcal{E}(\Pi_N h)\rangle\rangle\\
	=&\langle\langle\Phi,\mathcal{E}(\Pi_N h)\diamond\delta(-\mathcal{A}\Pi_Nh)\rangle\rangle\\
	=&\langle\langle\Phi,\mathcal{E}(\Pi_N h)\diamond\delta(\Pi_N(-\mathcal{A})h)\rangle\rangle\\
	=&\langle\langle\Phi,\Gamma(\Pi_N)(\mathcal{E}(h)\diamond\delta(-\mathcal{A}h))\rangle\rangle\\
	=&\langle\langle\Gamma(\Pi_N)\Phi,\mathcal{E}(h)\diamond\delta(-\mathcal{A}h)\rangle\rangle\\
	=&\langle\langle d\Gamma(-\mathcal{A})\Gamma(\Pi_N)\Phi,\mathcal{E}(h)\rangle\rangle\\
	=&\langle\langle d\Gamma(-\mathcal{A})\Phi_N,\mathcal{E}(h)\rangle\rangle.
\end{align*}
Comparing the first and last members of this chain of equalities we deduce that
\begin{align*} 
\Gamma(\Pi_N)d\Gamma(-\mathcal{A})\Phi=d\Gamma(-\mathcal{A})\Phi_N,\quad\mbox{ in $\mathtt{F}^{\star}$.}
\end{align*}
It is important to observe how in the fourth equality above the commutation between $-\mathcal{A}$ and $\Pi_N$ is made possible by having chosen to project onto the space generated by the eigenfunctions of $-\mathcal{A}$.\\
We now study the commutation between $\Gamma(\Pi_N)$ and $D^2_{1}$: 
\begin{align*}
	\langle\langle\Gamma(\Pi_N)D^2_{1}\Phi,\mathcal{E}(h)\rangle\rangle=&\langle\langle D^2_{1}\Phi,\mathcal{E}(\Pi_Nh)\rangle\rangle\\
	=&\langle\langle D_1D_1\Phi,\mathcal{E}(\Pi_Nh)\rangle\rangle\\
	=&\langle\langle \Phi,D_1^{\star}D_1^{\star}\mathcal{E}(\Pi_Nh)\rangle\rangle\\
	=&\langle\langle \Phi,\mathcal{E}(\Pi_Nh)\diamond\delta(1)\diamond\delta(1)\rangle\rangle\\
	=&\langle\langle \Phi,\Gamma(\Pi_N)(\mathcal{E}(h)\diamond\delta(1)\diamond\delta(1))\rangle\rangle\\
	=&\langle\langle \Gamma(\Pi_N)\Phi,\mathcal{E}(h)\diamond\delta(1)\diamond\delta(1)\rangle\rangle\\
	=&\langle\langle D_1D_1\Gamma(\Pi_N)\Phi,\mathcal{E}(h)\rangle\rangle\\
	=&\langle\langle D^2_{1}\Gamma(\Pi_N)\Phi,\mathcal{E}(h)\rangle\rangle.
\end{align*}
In the fifth equality above we employed the identity $\Pi_N1=\Pi_11=1$ since the first eigenfunction of $-\mathcal{A}$ is precisely $1$. We therefore can conclude that
\begin{align*}
	\Gamma(\Pi_N)D^2_{1}\Phi= D^2_{1}\Phi_N,\quad\mbox{ in $\mathtt{F}^{\star}$.}
\end{align*}
We proceed with the commutation between $\Gamma(\Pi_N)$ and $\mathtt{N}(\mathtt{N}-\mathtt{I})$:
\begin{align*}
\langle\langle\Gamma(\Pi_N)\mathtt{N}(\mathtt{N}-\mathtt{I})\Phi,\mathcal{E}(h)\rangle\rangle&=\langle\langle\mathtt{N}(\mathtt{N}-\mathtt{I})\Phi,\mathcal{E}(\Pi_Nh)\rangle\rangle\\
&=\langle\langle\Phi,\mathcal{E}(\Pi_Nh)\diamond\delta(\Pi_Nh)\diamond\delta(\Pi_Nh)\rangle\rangle\\
&=\langle\langle\Phi,\Gamma(\Pi_N)(\mathcal{E}(h)\diamond\delta(h)\diamond\delta(h))\rangle\rangle\\
&=\langle\langle\Phi_N,\mathcal{E}(h)\diamond\delta(h)\diamond\delta(h)\rangle\rangle\\
&=\langle\langle\Phi_N,\mathtt{N}(\mathtt{N}-\mathtt{I})\mathcal{E}(h)\rangle\rangle\\
&=\langle\langle \mathtt{N}(\mathtt{N}-\mathtt{I})\Phi_N,\mathcal{E}(h)\rangle\rangle.
\end{align*}
This yields
\begin{align*}
	\Gamma(\Pi_N)\mathtt{N}(\mathtt{N}-\mathtt{I})\Phi= \mathtt{N}(\mathtt{N}-\mathtt{I})\Phi_N,\quad\mbox{ in $\mathtt{F}^{\star}$.}
\end{align*}
Lastly,
\begin{align*}
	\langle\langle\Gamma(\Pi_N)D_{\lambda_c}^{\star}\Phi,\mathcal{E}(h)\rangle\rangle=&\langle\langle D_{\lambda_c}^{\star}\Phi,\mathcal{E}(\Pi_Nh)\rangle\rangle\\
	=&\langle\langle \Phi,\mathcal{E}(\Pi_Nh)\rangle\rangle\int_0^1\Pi_Nh(x)\lambda_c(x)dx\\
	=&\langle\langle \Gamma(\Pi_N)\Phi,\mathcal{E}(h)\rangle\rangle\int_0^1h(x)\Pi_N\lambda_c(x)dx\\
	=&\langle\langle \Phi_N,D_{\Pi_N\lambda_c}\mathcal{E}(h)\rangle\rangle\\
	=&\langle\langle D_{\Pi_N\lambda_c}^{\star}\Phi_N,\mathcal{E}(h)\rangle\rangle,
\end{align*}	
and hence
\begin{align*}
	\Gamma(\Pi_N)D_{\lambda_c}^{\star}\Phi=D_{\Pi_N\lambda_c}^{\star}\Phi_N,\quad\mbox{ in $\mathtt{F}^{\star}$.}
\end{align*}
An implementation of all derived identities in \eqref{proof projection} leads directly to \eqref{SDE 2}.
\end{proof}

\begin{remark}\label{commutation bound}
	The previous lemma is a key ingredient of our method since it allows for finite dimensional projections of the solution to equation \eqref{SDE}. In particular, it is the possibility of commuting $d\Gamma(-\mathcal{A})$ and $\mathtt{N}(\mathtt{N}-\mathtt{I})$ with $\Gamma(\Pi_N)$ that implies the desired result. It is worth to mention that such possibility exists because of Assumption \ref{constant}, thus motivating this strong simplification. In fact, without such assumption we would not be able to commute $\Gamma(\Pi_N)$ with the operator $\Phi\mapsto\delta^2(\lambda(\cdot,\cdot)D^2\Phi)$ which is what one should work with, in the place of $\mathtt{N}(\mathtt{N}-\mathtt{I})$, for non constant $\lambda_d$. One may also wonder whether changing the projection space related to $\Pi_N$ could solve this issue (maybe defining a finite dimensional space described by the function $\lambda_d$): however, this modification would imply the loss of commutativity between $\Gamma(\Pi_N)$ and $d\Gamma(-\mathcal{A})$.
\end{remark}

\begin{lemma}\label{partial}
	For any $N\geq 1$ there exists a function $u_N:[0+\infty[\times\mathbb{R}^N$ such that
	\begin{align*}
		\Phi_N(t)=u_N(t,I_1(\xi_1),...,I_1(\xi_N)),\quad \mathbb{P}\mbox{-a.s.}
	\end{align*}
	Furthermore, the function $u_N$ solves (weakly) the following fourth order linear problem:
	\begin{align}\label{PDE}
		\begin{cases}
			\partial_tu_N(t,z)=-\sum_{k=1}^N\alpha_k\partial^{\star}_k\partial_ku_N(t,z)+\frac{\lambda_d}{2}\partial_1^2u_N(t,z)\\
			\quad\quad\quad\quad\quad-\frac{\lambda_d}{2}\sum_{j,k=1}^N\partial_j^{\star}\partial_k^{\star}\partial_j\partial_ku_N(t,z)\\
			\quad\quad\quad\quad\quad+\sum_{k=1}^N\mathtt{c}_k\partial_k^{\star}u_N(t,z)-\gamma u_N(t,z);\quad t>0,z\in\mathbb{R}^N;\\
			u_N(0,z)=1,\quad z\in\mathbb{R}^N.
		\end{cases}
	\end{align} 
Here, for any $k\in\{1,...,N\}$ the symbol $\partial_k$ is a shorthand notation for $\partial_{z_k}$ while $\partial^{\star}_k$ stands for the differential operator $-\partial_k+z_k$ (which is nothing else that the Gaussian divergence). Moreover, $\mathtt{c}_k:=\langle\lambda_c,\xi_k\rangle_{L^2([0,1])}$.
\end{lemma}

\begin{proof}
It is well known (see for instance Theorem 4.9 in \cite{Janson}) that the second quantization operator $\Gamma(\Pi_N)$ corresponds to the conditional expectation with respect to the sigma-algebra generated by the random variables $I_1(\xi_1)$,..., $I_1(\xi_N)$; therefore, according to \eqref{Phi_N} we can write
\begin{align*}
	\Phi_N(t)=\Gamma(\Pi_N)\Phi(t)=\mathbb{E}[\Phi(t)|\sigma(I_1(\xi_1),..., I_1(\xi_N))]=u_N(t,I_1(\xi_1),...,I_1(\xi_N)).
\end{align*}  
Here, the function $u_N:[0+\infty[\times\mathbb{R}^N$ is measurable and its existence in guaranteed by Doob's lemma (see Lemma 1.13 in \cite{Kallenberg}). We now replace $\Phi_N(t)$ with $u_N(t,I_1(\xi_1),...,I_1(\xi_N))$ in \eqref{SDE 2} and decompose the Malliavin calculus's operators along the orthonormal bases $\{\xi_k\}_{k\geq 1}$. More precisely:
\begin{align*}
	d\Gamma(-\mathcal{A})[u_N(t,I_1(\xi_1),...,I_1(\xi_N))]=&\delta\left(-\mathcal{A}D[u_N(t,I_1(\xi_1),...,I_1(\xi_N))]\right)\\
	=&\delta\left(-\mathcal{A}\sum_{k\geq 1}D_{\xi_k}u_N(t,I_1(\xi_1),...,I_1(\xi_N))\xi_k\right)\\
	=&\delta\left(-\mathcal{A}\sum_{k=1}^N\partial_ku_N(t,I_1(\xi_1),...,I_1(\xi_N))\xi_k\right)\\
	=&\delta\left(-\sum_{k=1}^N\alpha_k\partial_ku_N(t,I_1(\xi_1),...,I_1(\xi_N))\xi_k\right)\\
	=&-\sum_{k=1}^N\alpha_k\partial_ku_N(t,I_1(\xi_1),...,I_1(\xi_N))I_1(\xi_k)\\
	&+\sum_{k=1}^N\alpha_k\partial_k^2u_N(t,I_1(\xi_1),...,I_1(\xi_N))\\
	=&-\sum_{k=1}^N\alpha_k\partial_k^{\star}\partial_ku_N(t,I_1(\xi_1),...,I_1(\xi_N));
\end{align*}	
Here, in the second-to-last equality we employed identity (1.56) from \cite{Nualart}. We proceed now with 
\begin{align*}
	D_1^2u_N(t,I_1(\xi_1),...,I_1(\xi_N))=\partial_1^2u_N(t,I_1(\xi_1),...,I_1(\xi_N)),
\end{align*}
and
\begin{align*}
	\mathtt{N}(\mathtt{N}-\mathtt{I})u_N(t,I_1(\xi_1),...,I_1(\xi_N))=&(\mathtt{N}^2-\mathtt{N})u_N(t,I_1(\xi_1),...,I_1(\xi_N))\\
	=&\sum_{j=1}^N\partial_j^{\star}\partial_j\left(\sum_{k=1}^N\partial_k^{\star}\partial_ku_N(t,I_1(\xi_1),...,I_1(\xi_N))\right)\\
	&-\sum_{k=1}^N\partial_k^{\star}\partial_ku_N(t,I_1(\xi_1),...,I_1(\xi_N))\\
	=&\sum_{j,k=1}^N\partial_j^{\star}\partial_k^{\star}\partial_j\partial_ku_N(t,I_1(\xi_1),...,I_1(\xi_N)).
\end{align*}
In the last equality we employed the commutation relation $\partial_j\partial_k^{\star}=\partial_k^{\star}\partial_j+\delta_{jk}\mathtt{I}$ where $\delta_{jk}$ stands for the Kronecker symbol. Lastly,
\begin{align*}
D^{\star}_{\Pi_N\lambda_c}u_N(t,I_1(\xi_1),...,I_1(\xi_N))=&\sum_{k\geq 1}D^{\star}_{\xi_k}u_N(t,I_1(\xi_1),...,I_1(\xi_N))\langle\Pi_N\lambda_c,\xi_k\rangle_{L^2([0,1])}\\
=&\sum_{k=1} ^N\partial^{\star}_ku_N(t,I_1(\xi_1),...,I_1(\xi_N))\langle\Pi_N\lambda_c,\xi_k\rangle_{L^2([0,1])}\\
=&\sum_{k=1} ^N\partial^{\star}_ku_N(t,I_1(\xi_1),...,I_1(\xi_N))\mathtt{c}_k.
\end{align*}
Collecting all identities derived above we see how equation \eqref{SDE 2} is equivalent to \eqref{PDE}.
\end{proof}

We are now ready to state the main result of the present section.
 
\begin{theorem}\label{main theorem}
	Let $\{\rho_n\}_{n\geq 0}$ be a classical solution to the CDME \eqref{equation2}-\eqref{initial}. Then, for any $N\geq 1$ we have the representation
	\begin{align}\label{formula}
		\Pi_N^{\otimes n}\rho_n(t,x_1,...,x_n)=\frac{1}{n!}\sum_{j_1,...,j_n=1}^N\mathbb{E}[(\partial_{j_1}\cdot\cdot\cdot\partial_{j_n}u_N)(t,I_1(\xi_1),...,I_1(\xi_n))]\xi_{j_1}(x_1)\cdot\cdot\cdot\xi_{j_n}(x_n),
	\end{align}
for all $n\geq 1$, $t\geq 0$, $(x_1,...,x_n)\in [0,1]^n$ and with $u_N$ solution to the Cauchy problem \eqref{PDE}.
\end{theorem}

\begin{proof}
If $\{\rho_n\}_{n\geq 0}$ is a classical solution to the CDME \eqref{equation2}-\eqref{initial}, then according to Lemma \ref{infinite} the stochastic process $\{\Phi(t)\}_{t\geq 0}$ defined in \eqref{Phi} solves equation \eqref{SDE} in $\mathtt{F}^{\star}$. Moreover, Lemma \ref{projection} shows that the finite dimensional projection of $\{\Phi(t)\}_{t\geq 0}$ introduced in \eqref{Phi_N} solves the auxiliary problem \eqref{SDE 2}. Notice that by construction the kernels of the Wiener It\^o chaos expansion of $\{\Phi_N(t)\}_{t\geq 0}$ are $\{\Pi_N^{\otimes}\rho_n\}_{n\geq 0}$ (since this is the action of $\Gamma(\Pi_N)$ on $\{\Phi(t)\}_{t\geq 0}$).   \\
On the other hand, according to the Stroock-Taylor formula (see Exercise 1.2.6 in \cite{Nualart}) the Wiener It\^o chaos expansion of $\{\Phi_N(t)\}_{t\geq 0}$ can also be represented as 
\begin{align}\label{ww}
	\Pi_N^{\otimes}\rho_n(t,x_1,...,x_n)=\frac{1}{n!}\mathbb{E}[D_{x_1}\cdot\cdot\cdot D_{x_n}\Phi_N(t)].
\end{align} 
From Lemma \ref{partial} the process $\{\Phi_N(t)\}_{t\geq 0}$ can be written as $\{u_N(t,I_1(\xi_1),...,I_1(\xi_n))\}_{t\geq 0}$ where $u_N$ solution to the Cauchy problem \eqref{PDE}. Therefore, substituting this into \eqref{ww} and computing the Malliavin derivatives yields
\begin{align*}
	\Pi_N^{\otimes}\rho_n(t,x_1,...,x_n)=&\frac{1}{n!}\mathbb{E}[D_{x_1}\cdot\cdot\cdot D_{x_n}u_N(t,I_1(\xi_1),...,I_1(\xi_n))]\\
	=&\frac{1}{n!}\sum_{j_1,...,j_n=1}^N\mathbb{E}[(\partial_{j_1}\cdot\cdot\cdot\partial_{j_n}u_N)(t,I_1(\xi_1),...,I_1(\xi_n))]\xi_{j_1}(x_1)\cdot\cdot\cdot\xi_{j_n}(x_n),
\end{align*}
which is the formula we wanted to prove.
\end{proof}

\subsection{Some comments on the PDE \eqref{PDE}}

The generalization of the moment generating function method utilized in this section has introduced some new Gaussian features to the original problem \eqref{equation2}-\eqref{initial}. At a formal level, the infinite dimensional nature of the system of Fokker-Planck equations under investigation combined with the Fock space structure of the sequence $\{\rho_n\}_{n\geq 0}$ leads naturally to the use of Gaussian stochastic analysis's techniques. \\
We now try to rewrite the representation formula \eqref{formula} in a Gaussian-free manner. To this aim, we present the following technical result.

\begin{lemma}\label{lemma GS}
	Let $f\in C^1(\mathbb{R}^N)$ be, together with all its first order partial derivatives, polynomially bounded at infinity. Then, setting
	\begin{align*}
		\tilde{f}(z):=\int_{\mathbb{R}^N}f(y)(2\pi)^{-N/2}e^{-|z-y|^2/2}dy,\quad z\in\mathbb{R}^N,
	\end{align*}   
we have for all $k\in\{1,...,N\}$ and $z\in\mathbb{R}^N$ that
\begin{align}\label{G to S}
	\widetilde{\partial_k f}(z)=\partial_k\tilde{f}(z)\quad\mbox{ and }\quad \widetilde{\partial_k^{\star} f}(z)=z_k\tilde{f}(z)
\end{align}
\end{lemma}

\begin{proof}
	It is a direct verification.
\end{proof}
		
\begin{proposition}
	If $u_N$ solves the PDE \eqref{PDE}, then $v_N:=\tilde{u}_N$ solves 
\begin{align}\label{PDE 2}
		\begin{cases}
			\partial_tv_N(t,z)=-\sum_{k=1}^N\alpha_kz_k\partial_kv_N(t,z)+\frac{\lambda_d}{2}\partial_1^2v_N(t,z)\\
			\quad\quad\quad\quad\quad-\frac{\lambda_d}{2}\sum_{j,k=1}^Nz_jz_k\partial_j\partial_kv_N(t,z)\\
			\quad\quad\quad\quad\quad+\sum_{k=1}^N\mathtt{c}_kz_kv_N(t,z)-\gamma v_N(t,z);\quad t>0,z\in\mathbb{R}^N;\\
			v_N(0,z)=1,\quad z\in\mathbb{R}^N.
		\end{cases}
	\end{align} 
\end{proposition}		
		
\begin{proof}
	Follows immediately from \eqref{G to S}.
\end{proof}		

The PDE \eqref{PDE 2} represents a version of \eqref{PDE} in which the Gaussian features inherited from our approach have been removed. Equation \eqref{PDE 2} has certainly the advantage over \eqref{PDE} of being of second order (contrary to the fourth order of the latter); moreover, if we consider the case $N=1$ and remember that $\xi_1\equiv 1$ and $\alpha_1=0$ we obtain
\begin{align}\label{book}
	\begin{cases}
		\partial_tv_1(t,z)=\frac{\lambda_d}{2}(1-z_1^2)\partial_1^2v_1(t,z)
+\gamma(z_1-1)v_1(t,z);\quad t>0,z\in\mathbb{R};\\
		v_1(0,z)=1,\quad z\in\mathbb{R}.
	\end{cases}
\end{align} 
This is exactly the equation you obtain via the classical moment generating function method applied to the CME \eqref{equation book}, which is the diffusion-free analogue of our system \eqref{equation2}-\eqref{initial}. See \cite{EC} for a detailed study of \eqref{book}. Therefore, from this point of view equation \eqref{PDE 2} is the natural extension of \eqref{book} to a model that includes diffusion of the particles.\\	
Even though equation \eqref{PDE 2} possesses some desirable properties, its investigation from both analytical and numerical points of view presents some important obstacles. First of all, if we use the function $v_N$ solution to \eqref{PDE 2} in the place of $u_N$ solution to \eqref{PDE}, then the representation formula \eqref{formula} takes the form
 \begin{align}\label{formula}
 	\Pi_N^{\otimes n}\rho_n(t,x_1,...,x_n)=\frac{1}{n!}\sum_{j_1,...,j_n=1}^N(\partial_{j_1}\cdot\cdot\cdot\partial_{j_n}v_N)(t,0,...,0)\xi_{j_1}(x_1)\cdot\cdot\cdot\xi_{j_n}(x_n),
 \end{align}
as it follows immediately by the definition of $v_N$ and Lemma \ref{lemma GS}. This means that the natural domain for solving \eqref{PDE 2} would be a neighborhood of the origin instead of the whole space; this can be seen already in the case $N=1$, i.e. equation \eqref{book}, where $z_1$ should be taken in $[-1,1]$ in order to avoid a sign change in the leading second order term. However, it is very hard to find a reasonable argument for assigning a boundary value to the problem \eqref{PDE 2} (this issue is also discussed in \cite{McQuarrie}).\\
A second main difficulty in analyzing equation \eqref{PDE 2} is due to its intrinsic ill-posedness. In fact, if for simplicity we take $N=2$ and focus on the second order (i.e. leading) term of the differential operator appearing in the right hand side of \eqref{PDE 2}, we see that the matrix describing its coefficients is a multiple of 
	\begin{align*}
		\begin{split} 
			A(z_1,z_2)=\begin{bmatrix} 1-z_1^2 & -z_1z_2 \\ -z_1z_2 & -z_2^2 
			\end{bmatrix}.
		\end{split}
	\end{align*}
Checking the positive semi-definiteness of the matrix $A$ (recall that where have an initial condition for solving equation \eqref{PDE 2}) we see that
\begin{align*}
	\langle A(z_1,z_2)\theta,\theta\rangle=&(1-z_1^2)\theta_1^2-2z_1z_2\theta_1\theta_2-z_2^2\theta_2^2\\
	=&\theta_1^2-(z_1\theta_1+z_2\theta_2)^2
\end{align*}
  and the last quantity cannot be non negative for any choice of $(\theta_1,\theta_2)\in\mathbb{R}^2$ unless $z_2=0$ (to see this take $\theta_1=0$). Therefore, there is no open neighborhood of the origin for the space variable $z$  where the matrix $A$ is positive semi-definiteness. This entails the ill-posedness of the PDE \eqref{PDE 2}.
  
  The discussion presented above highlights some potential advantages in embedding the CDME \eqref{equation2}-\eqref{initial} into the Gaussian framework utilized in this section for deriving the representation formula \eqref{formula}. 

\section*{Appendix}		

In this section we collect some definitions and formulas utilized in the proofs of Section 3. For more details on the subject we refer the reader to one the books \cite{Janson} and \cite{Nualart}.

\subsection*{Wiener chaos and spaces of random  variables}
Let $(\Omega,\mathcal{B},\mathbb{P})$ be the classical Wiener space over the interval $[0,1]$. We denote by 
\begin{align*}
	\begin{split}
		B_x&:\Omega\to\mathbb{R}\\
		&\quad\omega\mapsto B_x(\omega):=\omega(x),\quad x\in [0,1],
	\end{split}
\end{align*} 
the coordinate process which by construction is a one dimensional Brownian motion under $\mathbb{P}$. According to the Wiener-It\^o chaos expansion theorem, any random variable $\Phi$ in $\mathbb{L}^2(\Omega)$ can be uniquely represented as 
\begin{align*}
	\Phi=\sum_{n\geq 0}I_n(h_n),
\end{align*}  
where for $n\geq 1$, $I_n(h_n)$ stands for the $n$-th order multiple It\^o integral defined as 
\begin{align*}
	I_n(h_n):=n!\int_0^1\int_0^{x_1}\cdot\cdot\cdot\int_0^{x_{n-1}}h_n(x_1,...,x_n)dB_{x_n}\cdot\cdot\cdot dB_{x_2}dB_{x_1}.
\end{align*}
	Two notable dense subset of $\mathbb{L}^2(\Omega)$ are
	\begin{align*}
		\mathtt{F}:=\left\{\sum_{n=0}^MI_n(h_n),\mbox{ for some $M\in\mathbb{N}\cup\{0\}$, $h_0\in\mathbb{R}$ and $h_n\in L_s^2([0,1]^n)$, $n=1,...,M$}\right\}, 
	\end{align*}
	which collects the random variables with a finite order chaos expansion, and 
	\begin{align*}
		\mathtt{E}:=\left\{\mathcal{E}(f):=\sum_{n\geq 0}I_n\left(\frac{f^{\otimes n}}{n!}\right),\mbox{ for some $f\in L^2([0,1])$}\right\},
	\end{align*}
	which is the family of the so-called \emph{stochastic exponentials}.
	
	\subsection*{Malliavin derivative and its adjoint}
	
	The \emph{Malliavin derivative} of $\Phi=\sum_{n=0}^MI_n(h_n)\in\mathtt{F}$, denoted $\{D_x\Phi\}_{x\in[0,1]}$, is the element of $L^2([0,1];\mathtt{F})$ defined by
	\begin{align*}
		D_x\Phi:=\sum_{n=0}^{M-1}(n+1)I_n(h_{n+1}(\cdot,x)), \quad x\in [0,1].
	\end{align*}
	For $l\in L^2([0,1])$ and $\Phi=\sum_{n=0}^MI_n(h_n)\in\mathtt{F}$, we also write
	\begin{align*}
		\begin{split}
			D_l\Phi:=\langle D\Phi,l\rangle_{L^2([0,1])}&=\sum_{n=0}^{M-1}(n+1)I_{n}\left(\int_0^1h_{n+1}(\cdot,y)l(y)dy\right)\\
			&=\sum_{n=0}^{M-1}(n+1)I_{n}\left(h_{n+1}\otimes_1 l\right)
		\end{split}
	\end{align*}
	for the \emph{directional Malliavin derivative} of $\Phi$ along $l$. Here, we denote the \emph{$r$-th order contraction} of $h_n$ and $h_m$ by $h_n\otimes_r h_m$, i.e.
	\begin{align*}
		\begin{split}
			&(h_n\otimes_r h_m)(x_1,....,x_{n+m-2r})\\
			&:=\int_{[0,1]^r}h_n(x_1,...,x_{n-r},y_1,...,y_r)h_m(y_1,...,y_r,x_{n-r+1},...,x_{n+m-2r})dy_1\cdot\cdot\cdot dy_r.
		\end{split}
	\end{align*}
We have:
\begin{align*}
	D_x\mathcal{E}(f)=f(x)\mathcal{E}(f), x\in [0,1]\quad\mbox{ and }\quad D_l\mathcal{E}(f)=\langle f,l\rangle_{L^2([0,1])}\mathcal{E}(f).
\end{align*} 
If we now take $l\in L^2([0,1])$, $\Phi=\sum_{n=0}^MI_n(h_n)\in\mathtt{F}$ and $\Psi=\sum_{n=0}^KI_n(g_n)\in\mathtt{F}$, we can write
\begin{align*}
	\mathbb{E}[D_l\Phi\cdot\Psi]=\mathbb{E}[\Phi\cdot D_l^{\star}\Psi],
\end{align*}
where
\begin{align*}
	D^{\star}_l\Psi:=\sum_{n=1}^{K+1}I_n(l\hat{\otimes}g_{n-1})
\end{align*}
and
\begin{align*}
	(l\hat{\otimes}g_{n-1})(x_1,...,x_n):=\frac{1}{n}\sum_{i=1}^nf(x_i)g_{n-1}(x_1,...,x_{i-1},x_{i+1},...,x_n).
\end{align*}
	The following identity holds:
	\begin{align*}
		D_l^{\star}\Psi+D_l\Psi=\Psi\cdot I_1(l).
	\end{align*}
 One can also introduce the adjoint of $D_x$, denoted $\delta$: 
\begin{align*}
	\delta(\Phi(\cdot)):=\sum_{n=0}^MI_{n+1}(\tilde{h}_n)\in\mathtt{F},
\end{align*}
where $\tilde{h}_n$ stands for the symmetrization of $h_n$ with respect to the $n+1$ variables $x_1,...,x_n,x$. We mentioned that $D^{\star}_l\Psi$ it is sometimes written as $\Phi\diamond\delta(l)$.

\subsection*{Second quantization operators}

Let $A:L^2([0,1])\to L^2([0,1])$ be a bounded linear operator; for $\Phi=\sum_{n=0}^MI_n(h_n)\in\mathtt{F}$ we define the \emph{second quantization operator} of $A$ as
\begin{align*}
	\Gamma(A)\Phi:=\sum_{n=0}^MI_n\left(A^{\otimes n}h_n\right),
\end{align*}
and the \emph{differential second quantization operator} of $A$ as
\begin{align*}
	d\Gamma(A)\Phi:=\sum_{n=1}^MI_n\left(\sum_{i=1}^nA_ih_n\right),
\end{align*}
where $A_i$ stands for the operator $A$ acting on the $i$-th variable of $h_n$.  Notice in addition that for $A$ being the identity, we recover from $d\Gamma(A)$ the well known \emph{number operator}:
\begin{align*}
	\mathtt{N}\Phi=\sum_{n=1}^MnI_n\left(h_n\right).
\end{align*}
The following identities hold true:
\begin{align*}
	\begin{split}
		&\mathbb{E}[\Gamma(A)\Phi]=\mathbb{E}[\Phi];\quad \mathbb{E}[d\Gamma(A)\Phi]=0;\\
		&\mathbb{E}[\Gamma(A)\Phi\cdot\Psi]=\mathbb{E}[\Phi\cdot\Gamma(A^{\star})\Psi];\quad\mathbb{E}[d\Gamma(A)\Phi\cdot\Psi]=\mathbb{E}[\Phi\cdot d\Gamma(A^{\star})\Psi];\\
		&\Gamma(A)\mathcal{E}(f)=\mathcal{E}(Af);\quad d\Gamma(A)\mathcal{E}(f)=D^{\star}_{Af}\mathcal{E}(f);\quad d\Gamma(A)\Phi=\delta\left(AD_{\cdot}\Phi\right).
		\end{split}
\end{align*} 

\subsection*{The space  $\mathtt{F}^{\star}$}

Let  
\begin{align*}
	\mathtt{F}^{\star}:=\left\{\sum_{n\geq 0}I_n(h_n),\mbox{ for some $h_0\in\mathbb{R}$ and $h_n\in L_s^2([0,1]^n)$, $n\geq 1$}\right\} 
\end{align*}
be a family of \emph{generalized} random variables. The action of $T=\sum_{n\geq 0}I_n(h_n)\in\mathtt{F}^{\star}$ on $\varphi=\sum_{n=0}^MI_n(g_n)\in\mathtt{F}$ is defined as 
\begin{align*}
	\langle\langle T,\varphi\rangle\rangle:=\sum_{n=0}^Mn!\langle h_n,g_n\rangle_{L^2([0,1]^n)}.
\end{align*}  
By construction, we have the inclusions
\begin{align*}
	\mathtt{F}\subset\mathbb{L}^2(\Omega)\subset\mathtt{F}^{\star}
\end{align*}
with
\begin{align*}
	\langle\langle T,\varphi\rangle\rangle=\mathbb{E}[T\varphi],
\end{align*}
whenever $T\in\mathbb{L}^2(\Omega)$. We will say that $T=U$ in $\mathtt{F}^{\star}$ if
\begin{align*}
	\langle\langle T,\varphi\rangle\rangle=\langle\langle U,\varphi\rangle\rangle, \quad \mbox{for all }\varphi\in\mathtt{F}.
\end{align*}
 Let  
\begin{align*}
	\mathtt{F}^{\star}:=\left\{\sum_{n\geq 0}I_n(h_n),\mbox{ for some $h_0\in\mathbb{R}$ and $h_n\in L_s^2([0,1]^n)$, $n\geq 1$}\right\} 
\end{align*}
be a family of \emph{generalized} random variables. The action of $T=\sum_{n\geq 0}I_n(h_n)\in\mathtt{F}^{\star}$ on $\varphi=\sum_{n=0}^MI_n(g_n)\in\mathtt{F}$ is defined as 
\begin{align*}
	\langle\langle T,\varphi\rangle\rangle:=\sum_{n=0}^Mn!\langle h_n,g_n\rangle_{L^2([0,1]^n)}.
\end{align*}  
By construction, we have the inclusions
\begin{align*}
	\mathtt{F}\subset\mathbb{L}^2(\Omega)\subset\mathtt{F}^{\star}
\end{align*}
with
\begin{align*}
	\langle\langle T,\varphi\rangle\rangle=\mathbb{E}[T\varphi],
\end{align*}
whenever $T\in\mathbb{L}^2(\Omega)$. We will say that $T=U$ in $\mathtt{F}^{\star}$ if
\begin{align*}
	\langle\langle T,\varphi\rangle\rangle=\langle\langle U,\varphi\rangle\rangle, \quad \mbox{for all }\varphi\in\mathtt{F}.
\end{align*}
Lastly, we recall a generalized version of the so-called Stroock-Taylor formula: if $T=\sum_{n\geq 0}I_n(h_n)\in\mathtt{F}^{\star}$, then
\begin{align*}
	h_n(x_1,...,x_n)=\frac{1}{n!}\mathbb{E}[D_{x_1}...D_{x_n}T], \quad (x_1,...,x_n)\in [0,1]^n.
\end{align*}
	
\bibliographystyle{apalike}
\bibliography{mutual_annihilation}

\end{document}